\documentclass[reqno]{amsart}
 
\usepackage{amsfonts} 
\usepackage{amsmath}
\usepackage{amssymb}
\usepackage{amsthm}
\usepackage{mathtools}

 \usepackage{bm}  
\usepackage{empheq}
\usepackage{mathptmx}

\allowdisplaybreaks

\usepackage{enumitem}
 
\usepackage{wasysym}
 
\usepackage{verbatim}
 
\usepackage{graphicx}

\usepackage{microtype}
\theoremstyle{plain}\newtheorem{theorem}{Theorem}[section]

\newtheorem{lemma}[theorem]{Lemma}
\newtheorem{problem}[theorem]{Problem}
\newtheorem{proposition}[theorem]{Proposition}

\renewenvironment{proof}[1][Proof]{\textbf{#1.} }{\ \rule{0.5em}{0.5em} \par }
\theoremstyle{remark}

\theoremstyle{definition}
\newtheorem{remark}[theorem]{Remark}
\newtheorem{definition}[theorem]{Definition}

\def\PP{\mathbb{P}}
\def\RR{\mathbb{R}}
\def\EE{\mathbb{E}}

\def\be{{\beta}}

\def\la{{\lambda}}
\def\si{{\sigma}}

\def\al{{\alpha}}

\def\be{{\beta}}
\def\Ga{{\Gamma}}
\def\ga{{\gamma}}

\def\si{{\sigma}}

\def\la{{\lambda}}

\def\vare{{\varepsilon}}

\def\si{{\sigma}}
\def\al{{\alpha}}

\newcommand{\laa}{\Lambda}
\newcommand{\oom}{\Omega}

\newcommand{\vp}{\varphi}
\def\eref#1{\hbox{(\ref{#1})}}

\newcommand{\lp}{\left(}
\newcommand{\rp}{\right)}
\newcommand{\lc}{\left[}
\newcommand{\rc}{\right]}

\newcommand{\R}{{\mathbb R}}

\newcommand{\HH}{\mathbb H}

\newcommand{\ep}{\ensuremath{\varepsilon}}

\newcommand{\upchi}{\text{\usefont{U}{psy}{m}{n}\symbol{'143}}}

\let\Section=\section
\def\section{\setcounter{equation}{0}\Section}

\begin{document}

\title[\bf Parabolic Anderson  equation]{\bf    
 Solvability of 
 parabolic Anderson  equation
with  fractional Gaussian noise 
} 
\author{\sc Zhen-Qing Chen  }
\address{Department of Mathematics, University of Washington, Seattle, WA 98195, USA }
  \email{zqchen@uw.edu }
\author{ \sc   Yaozhong Hu}
\address{Department of Mathematical and  Statistical
 Sciences, University of Alberta at Edmonton,
Edmonton, Canada, T6G 2G1}
  \email{ yaozhong@ualberta.ca
}
\thanks{The research of ZC is supported in part by a  Simons Foundation Grant. 
The research of YH is supported in part by an NSERC Discovery grant  and a startup fund from University of Alberta at Edmonton.} 
 
\date{}  
 
%These are the same as MSC 2020 but don't know how to change 2010 to 2020.


\begin{abstract} This paper   provides  necessary   as well as    sufficient
conditions on the Hurst parameters so that the continuous time
parabolic Anderson  model $\frac{\partial u}{\partial t}=\frac{1}{2}\frac{\partial^2 u}{\partial x^2}+u\dot{W}$ 
  on $[0, \infty)\times \R^d  $  with $d\geq 1$  has a unique  random field   solution, where $W(t, x)$ is a fractional Brownian sheet
  on $[0, \infty)\times \R^d$ and formally $\dot W =\frac{\partial^{d+1}}{\partial t \partial x_1 \cdots \partial x_d} W(t, x)$.
When the noise $W(t, x)$ is white in time, our condition is both necessary and sufficient
 when the initial data $u(0, x)$ is bounded between two positive constants. When the noise 
is fractional in time with Hurst parameter $H_0>1/2$, 
   our sufficient condition, which improves the known results in literature,
  is  different from  the necessary one.    
  \end{abstract} 

\maketitle

\section{Introduction}

Let $d\geq 1$. 
In this paper we are interested in the  following  stochastic  heat    equation (Parabolic Anderson model) 
for   $u=u(t, x)$ with $t\geq 0$ and $x=(x_1, \cdots, x_d) \in{\mathbb{R}}^d$:   
\begin{equation}\label{spde}
\frac{\partial u}{\partial t}=\frac{1}{2}\frac{\partial^2 u}{\partial x^2}+u\dot{W}\,  \quad \hbox{on } 
(0, \infty) \times \R^d,
\end{equation}
   where {  $W= \left\{ W(t, x); \, t\in [0, \infty), \, x\in \R^d \right\}$ }
   is a centered Gaussian process defined  on some complete probability space 
  $(\Omega, \mathcal{F}, \PP)$ with  the  covariance given
by
\begin{equation}  \label{cov}
\EE \lc W(s,x)W(t,y)\rc= R_{H_0}(s,t) \prod_{i=1}^dR_{H_i}(x_i,y_i)    , 
\quad s, t\geq 0, \, x, y \in \R^d, 
\end{equation}
where $\frac12\le H_0<1$ and  $ 0<H_i <1$ for $ i=1,\cdots, d$,   and
where  for  $\beta >0$,
\begin{equation}
 R_{\beta }(a, b) :=
\frac{1}{2 } \lp |a|^{2\beta }+|b|^{2\beta }-|a -b |^{2\beta }\rp  	\quad \hbox{for all }  a, b\in \RR\,.
\end{equation}    
Here we use $:=$ as a way of definition.    That is, $W(t, x)=W(t, x_1, \cdots, x_d)$ is a fractional Brownian motion of Hurst parameter $H_0\ge 1/2$    in time variable and  is a 
fractional Brownian motion with Hurst parameter $ H=(H_1, \cdots, H_d)$  in     space variables,  and formally 
$\dot{W}(t, x) =\frac {\partial ^{d+1}W}{\partial t \partial x_1\cdots\partial x_d}$. 
 Note that  $R_{1/2}(s,t)=s\wedge t$  and so   $\frac{\partial ^2}{\partial s \partial  t}
R_{1/2}(s,t) =\delta_{\{0\}}(s-t)$ in the   distributional sense, where $\delta_{\{0\}}$ is the Dirac measure on $\R$ concentrated at the origin $0$. 
  One dimensional Brownian motion has covariance function $R_{1/2}$ so it has Hurst parameter 1/2. 
 For fractional Brownian motions, the smaller the Hurst parameters are, the rougher their sample paths.
  The precise meaning of the (random field) solution  to   
 \eqref{spde}  will be given  later in Definition \ref{def-sol-sigma}. 
The product $u\dot{W}$ in the above  equation is in the sense of Wick, which  means that 
in Equation \eqref{e.3.1} of the definition of the random field  solution 
the stochastic integral is understood as the 
  Ito-Skorohod   integral. 
  
The equation \eqref{spde} is one of the simplest stochastic partial differential equations
and describes a heat propagation through a random medium.  It has a close connection with KPZ equation
through the Hopf-Cole transform.  Equation \eqref{spde}
has been studied by many authors. 
We refer the reader to a recent survey \cite{Hu} and references therein.

  Throughout the paper,  we assume that the Hurst parameter $H_0$ in time is always greater than or equal to $1/2$.  But some of the spatial   Hurst parameters $H_i$    in \eqref{spde}  can be  less than $1/2$ while others  are 
greater than  or equal to $1/2$.  Let $d_*$ denote the total number of  $H_i$ whose value is strictly less than 1/2.
Without loss of generality and for the simplification of notation we can assume that  $H_k<1/2$ for $1\leq k\leq d_*$    and
and $H_k\ge 1/2$ for $d_* <  k \leq d$ for some $d_* \in \{0, 1, \cdots, d\}$. 
    Let
$$
 H_*=H_1+\cdots+H_{d_*} \quad \hbox{and} \quad  { H^* =H_1+\cdots+H_d } .
$$
  It is known      that when $H_0\ge 1/2$
  and    $H_i>1/2$ for all $1\leq i\leq d$,  
the equation \eqref{spde} has a unique random field solution when $H^*>d-1$
(see  \cite[Example 2.6]{HHNT}; see also    \cite{Dal}  for more general Gaussian noises that is white in time, including
the fractional Brownian noise with 
 $H_0=1/2$ and all the spatial Hurst parameters being  greater than $1/2$). 
   When $H_0=1/2$ and $d_*=d$,  it is a folklore that $d$ must be 1 
     for \eqref{spde}  to be solvable  
   with bounded initial value. 
  It is shown in 
 \cite{HHLNT1, HHLNT2} that when $d=1$, $H_0=1/2$ and  $1/4<H^*<1/2$, 
the   equation \eqref{spde} has a unique random field solution.     

Xia Chen considered  fractional Gaussian noise $W$  
that has $H_0\in [1/2, 1)$   but some of  $H_i\in (0, 1)$,  $1\leq i\leq d$, 
are  greater than or equal to $1/2$ while others are allowed to be less than  $1/2$. 
 He showed   \cite[Theorems 1.2 and  1.3]{chen}   that when 
 \begin{equation}\label{e:1.3} 
\begin{cases}
2(d-{H^*} )+(d_*-2H_*)<2 &\quad \hbox{when } H_0=1/2,  \smallskip \\
   {H^*} >d-1 \, \hbox{ and }  \
4(1-H_0)+2(d-{H^*} )+(d_*-2H_*)<4   &\quad \hbox{when } H_0>1/2 \,, 
\end{cases}
 \end{equation} 
 the SPDE  \eqref{spde} has a unique global random field  solution in   $L^2(\Omega, \mathcal{F},  \PP)$
for any initial value $u_0(x)$ that is bounded.
  He also showed that, for $d\geq 2$,  when
 \begin{equation}\label{e:1.3b} 
   H_0>1/2, \quad  {H^*} =d-1 \quad  \hbox{and}  \quad
4(1-H_0)+ (d_*-2H_*)<2   \,, 
  \end{equation} 
the SPDE  \eqref{spde} has a unique local random field  solution in   $L^2(\Omega, \mathcal{F},  \PP)$
for any initial value $u_0(x)$ that is bounded. 
 An interesting and challenging    
problem is  whether  the above conditions are also necessary.   As we shall see,  the 
answer will be no and when $H_0=1/2$    we  give a 
necessary and sufficient condition for the existence of solution to \eqref{spde},
and when $H_0>1/2$, we improve    the sufficient condition \eqref{e:1.3} as well as give a necessary condition.   

  The following  are  the two main results of this paper, by considering $H_0=1/2$ and $1/2< H_0<1$ separately.     
  See Definition \ref{def-sol-sigma}  for the  precise definition of global and local solution to \eqref{spde}. 
   Our main results show that   the condition      $H^*>d-1$ is sufficient for the existence of   global 
    random field solution
  to \eqref{spde} 
  when  $d\geq 2$ and $H_0\in [1/2, 1)$, and is also necessary   for the existence of local random field  solution to \eqref{spde}  
  when $H_0=1/2$.  
  We further show that when $H_0\in (1/2, 1)$ and $d\geq 2$,  there exists  a unique local solution in the critical case $H^*=d-1$.
 Our necessary condition for  the case of $H_0>1/2$  
 is different  from the sufficient condition 
 but it  involve $H^*$ and $H_0$ only; 
 we do not need to separate the rougher ones (those with $H_i < 1/2$)
   from the smoother ones (those with $H_i \geq 1/2$) in spatial Hurst parameters.

\begin{theorem}\label{t.1.1}  Suppose  $H_0=1/2$; that is, the noise $  W$ is white in time. 
\begin{enumerate} 
\item[\rm (i)]   Suppose the initial condition satisfies $|u_0(x)|\le C$ for some  constant
$C>0$. If  
\begin{equation}\label{e:1.4}
\begin{cases}
{H^*} >1/4&\qquad\hbox{when }d= 1,  \smallskip  \\
{H^*} >d-1&\qquad\hbox{when  }d\ge 2,
\end{cases}  
\end{equation}
then the equation \eqref{spde} has a unique (global)  random field  solution in $L^2(\Omega, \mathcal{F},  \PP)$
with $u(0, x)=u_0 (x)$. 
Moreover, in this case, there is a positive constant $C_H>0$ so  that 
\begin{equation}
\EE \left[ u(t,x)^{  p}\right]\le C_H \exp\left( C_H t p^{\frac{{H^*} -d+2}{{H^*} -d+1}}\right) 
\quad \hbox{for any }  t\ge 0 \hbox{ and }  \ p\ge 2\,. \label{e.1.5} 
\end{equation}

\item[\rm (ii)]  Let the initial condition  satisfy 
$u_0(x)\ge c$  (or $u_0(x)\le -c$)  for some positive  constant $c>0$. 
 If    the equation \eqref{spde} has a   local    solution in $L^2(\Omega, \mathcal{F},  \PP)$,   then \eqref{e:1.4} must be satisfied. 
 \end{enumerate} 
\end{theorem}

\begin{remark}\label{r.1.2}  
	\begin{enumerate}
		\item[(i)] In the case of $H_0=1/2$, the condition \eqref{e:1.3} is equivalent to \eqref{e:1.4} when $d=1$ and 
		is stronger than  \eqref{e:1.4} when $d\ge 2$ 
		  since $d_*-2H_*  > 0$ when $d_*\geq 1$ and  $d_*-2H_* = 0$ when $d_*=0$.  
	 	\item[(ii)]  Since   $0<H_i<1$ for all $1\leq i=\leq d$ and $1\leq H_i<1/2$ for $1\leq i\leq d_*$,
		we have $H^*< \frac{d_*}{2} +(d-d_*)$,  that is, $H^*< d- (d^*/2)$. 
		  Thus,  Condition  \eqref{e:1.4} implies  that 
		$d- (d_*/2)>d-1$, or  $d_*<2$.  Thus under condition \eqref{e:1.4},  $d_*$ can only be $0$ or $1$.  
	\end{enumerate}
 
\end{remark}

\begin{theorem}\label{t.1.2}  Suppose   $H_0>1/2$.
\begin{enumerate} 
\item[\rm (i)]  
  Let {  $d\ge 1$.}  
   Suppose the initial condition satisfies $|u_0(x)|\le C$ for some  constant
$C>0$. If  
\begin{equation}\label{e:1.7} 
{ 
\begin{cases}
{H^*} >\frac{3}4 -H_0 &\qquad\hbox{when }d= 1,  \smallskip  \\
{H^*} >d-1&\qquad\hbox{when  }d\ge 2,
\end{cases} 
}
 \end{equation}
then the equation \eqref{spde} has a unique (global) random field  solution in $L^2(\Omega, \mathcal{F},  \PP)$
with $u(0, x)=u_0 (x)$. 
Moreover, in this case, there is a positive constant $C_{H,d} >0$ so  that 
\begin{equation}
\EE \left[ u(t,x)^p\right]\le 
C\exp\left[ C_H t^{\frac{{H^*} -d+2H_0}{{H^*} -d+1}} p^{\frac{{H^*} -d+2}{{H^*} -d+1}}
 \right]\quad \hbox{for any }  t\ge 0  \hbox{ and }  p\ge 2\,. \label{e.1.8} 
  \end{equation}
Moreover, if 
\begin{equation}\label{e:1.7b} 
d\geq 2 \quad \hbox{and} \quad  H^*=d-1, 
\end{equation} 
the equation \eqref{spde} has a unique local random field  solution in $L^2(\Omega, \mathcal{F},  \PP)$
with $u(0, x)=u_0 (x)$.

\smallskip

\item[\rm (ii)]  Let the initial condition  satisfy 
$u_0(x)\ge c$ for some positive  constant $c>0$. 
 If    the equation \eqref{spde} has a    local  solution in $L^2(\Omega, \mathcal{F},  \PP)$,  then 
\begin{equation}  \label{e.1.9}
 H^*+2H_0 >
\begin{cases}
 5/4 &\qquad \hbox{when } d=1,  \smallskip \\
(3d+2)/{4} &\qquad \hbox{when }d\ge 2\,. \\
\end{cases}  
\end{equation}
 \end{enumerate} 
\end{theorem}  

\begin{remark} 
\begin{enumerate}  
\item[(i)]   When $d=1$ and $H_0>1/2$, condition \eqref{e:1.3}
is equivalent to $H^*>\frac34 -H_0$.  
Clearly,  when $H_0>1/2$, conditions  \eqref{e:1.3} and \eqref{e:1.3b}
are  stronger than   conditions \eqref{e:1.7} and \eqref{e:1.7b}, respectively, when  $d\geq 2$.  

\smallskip
 
\item[(ii)] Condition \eqref{e:1.4} is the same as \eqref{e:1.7} if we take $H_0=1/2$ there.

\smallskip

\item[(iii)]   The necessity of the condition \eqref{e:1.4} when $H_0=1/2$ 
and of the condition \eqref{e.1.9} when $H_0>1/2$
for the existence of solutions to SPDE \eqref{spde} 
is  new.   
There seems no discussions  about the necessary conditions 
for this equation before.  

\smallskip

\item[(iv)]   Note that the sufficient condition and the necessary condition in Theorems \ref{t.1.1} and \ref{t.1.2} have different 
requirements on the initial value $u(0, x)$ of $u$. 
It is easy to check directly that, when $1/2<H_0<1$,  the sufficient condition \eqref{e:1.7} is strictly stronger than the necessary condition \eqref{e.1.9}.   

\smallskip

\item[(v)]  Estimate   \eqref{e.1.8} coincides with the upper bound part of (6.1) of
\cite{HHNT} when all Hurst parameters are greater than $1/2$   (by setting    $\eta_i=2H_i-2$ and $\beta=2H_0-2$), so we expect our bound 
\eqref{e.1.8} is sharp. 

\smallskip

  \item[(vi)] As we see from  Remark \ref{r.1.2}(ii) that $H^*<d- (d_*/2)$. So condition \eqref{e:1.7b}  implies that 
  $d_*<2$, that is, either $d_*=0$ or $d_*=1$.  
  
\smallskip

  \item[(vii)]  Again from the fact that $H^*<d- (d^*/2)$, we see that when $d\geq 2$,
condition \eqref{e.1.9} implies  $d_*<(d/2)   +4H_0 -1< (d+6)/2$.  But  this  
    condition is not optimal.  In fact, when $W$ is time indenpendent (which corresponds to the case $H_0=1$) and 
space white (which corresponds to the case $H_1=\cdots=H_d=1/2$),  it is known     
(see \cite{hu02})   
that the  equation  \eqref{spde} has a global solution   when $d=1$ and has   a local solution  but has no global solution   when $d= 2$,
 and has  no
any local  solution when $d\ge 3$. So our conjecture is that  even when $H_0>1/2$ and $H_i\in [1/2, 1)$ for all $1\leq i\leq d$, 
to ensure the global unique solution we need  $d_*\le 1$ and to ensure local solution 
we need $d_*\le 2$.

 \end{enumerate} 
 \end{remark} 
 
In this paper, we do not discuss the solvability of \eqref{spde} when  the   time   Hurst  parameter $H_0\in (0, 1/2)$.
  We refer the interested reader  to \cite{CHKN,chen2, HLN} and references therein
  for recent development when $H_0<1/2$.  Let us also mention that for the additive noise
  (namely, replace $u\dot W$ by $\dot W$ in \eqref{spde}, the necessary 
  and sufficient condition is   known (\cite{HLT})  even for more general Gaussian noise.    
 
 \medskip
 
The rest of the paper is organized as follows. 
In Section \ref{sec:preliminaries}, we recall some facts on Gaussian random fields,   stochastic integrals with respect to them
and their properties 
that will be used in this paper.  
The proofs of  the sufficient part of Theorems \ref{t.1.1} and \ref{t.1.2} are presented in   Section \ref{S:3},
while the proof of the necessary part of these two theorems are  given in Section \ref{S:4}.
In this paper, for  $a, b\in \R$, $a\vee b:= \max \{a, b\}$ and $a\wedge b :=\min \{a, b\}$.

\section{Preliminaries}\label{sec:preliminaries}

The noise $W$ can be viewed as a Brownian motion with values in an infinite dimensional Hilbert space. One might thus think that the stochastic integration theory with respect to $W$ can be handled by classical theories  (e.g. \cite{DPZ}). However, the spatial covariance function of $\dot W$, which is formally   $\prod_{i=1}^d H_i(2H_i-1) |x_i-y_i|^{2H_i-2}$,  is  not  locally integrable  along
 the diagonals as  $H_i<1/2$ for $1\leq i\leq   d_*$.  The  
stochastic integral with respect to $W$ then needs to be dealt with 
through    other means. We recall briefly some 
key points needed in this  paper and we refer to  \cite{HHLNT1,HHLNT2}
for more details.

We start by introducing some  basic notation on Fourier transforms. The space of   Schwartz functions is
denoted by  $\mathcal{S}$. Its dual, the space of tempered distributions, is 
denoted by  $\mathcal{S}'$.

The Fourier
transform of a function $u \in \mathcal{S} $ is defined  by 
\[
\widehat  u(\xi) :=\int_{\RR^d} e^{-i\xi x} u(x) dx\ 
\] 
so that the inverse Fourier transform is given by 
  $$
 \check u (\xi) :=  ( 2 \pi)^{- d} \int_{\RR^d} e^{i\xi x} u(x) dx =   ( 2 \pi)^{- d} \widehat u (-\xi) .
 $$ 

 Let $C^\infty_c((0,\infty)\times \R^d)$ denote the space  of real-valued infinitely differentiable functions with compact support on $(0, \infty) \times \R^d$. The noise  $W$  can be described  by a mean zero Gaussian family $\{W(\vp) ,\, \vp\in
C^\infty_c((0,\infty)\times \R^d)\}$ defined on a complete probability space
$(\Omega,{\mathcal{F}} ,\mathbb{P})$, whose covariance structure
is given by
\begin{equation}\label{eq:cov1a}
\EE\lc W(\vp) \, W(\psi) \rc
=  c_{1,H}\int_{\R_{+}^2\times\R^d}
 \widehat \varphi(s,\xi) \, \overline{\widehat \psi}   (r,\xi)
\, \prod_{i=1}^d |\xi_i|^{1-2H_i} \, \ga_0(s-r) dsdr   d\xi,
\end{equation}
where 
\begin{equation}\label{e:2.2} 
\ga_0(s-r) =H_0(2H_0-1) |s-r|^{2H_0-2}.
\end{equation}
  When $H_0=1/2$, we replace $\ga_0(s-r)$ by the 
Dirac delta function $\ga_0(s-r)=\delta_{\{0\}}(s-r)$: 
\begin{equation}\label{eq:cov1b}
\EE\lc W(\vp) \, W(\psi) \rc
=  c_{1,H}\int_{\R_{+}\times\R^d}
 \widehat \varphi(s,\xi) \, \overline{\widehat \psi}   (s,\xi)
\, \prod_{i=1}^d |\xi_i|^{1-2H_i} \, ds  d\xi,
\end{equation}
where the Fourier transforms $\widehat \varphi$ and $\widehat   \psi$ are understood as Fourier transforms in spatial variables  only and
\begin{equation}\label{eq:expr-c1H}
c_{1,H}= \frac 1 {(2\pi)^d } \prod_{i=1}^d \Gamma(2H_i+1)\sin(\pi H_i)  \,.
\end{equation}
  
As above throughout the remaining part of the paper we will always replace $\ga_0(s-r)$  
by $\delta_{\{0\}}(s-r)$ 
when $H_0=1/2$.
  
 Let  
 $$
 \HH_0=\left\{ \varphi  \in \mathcal{S}: \  \| \varphi \|_{\HH}:=
 \left( \int_{\RR_+^2\times \RR^d} 
     \widehat \varphi
 (s,\xi) \overline{\widehat \varphi}
 (r,\xi) |\xi|^{1-2H}d\xi \ga_0(s-r) ds dr \right)^{1/2} < \infty \right\}.
 $$
Since $\ga_0(s-r)$ is a positive definite kernel it  is well-known that 
$\|\cdot\|_\HH$ is a Hilbert norm (in fact it is the  $L^2$ norm of the 
stochastic integral $\int_{\RR_+\times \RR^d}  \varphi
 (s,x)W(ds,dx)$).  
   Let $ \HH$  be the completion of 
 $\HH_0$ under the above norm $\| \cdot \|_{\HH}$.   Using this Hilbert norm   
  we can define  the stochastic integration with respect to $W$. 
 
\begin{definition}\label{def:elementary-process}
For any $t\ge0$, let $\mathcal{F}_{t}$ be the $\sigma$-algebra generated by $W$ up to time $t$. An elementary process $g$ is a process given by
\begin{equation}
g(s,x)
=
\sum_{i=1}^{n} \sum_{j=1}^m X_{i,j} \, \upchi_{(a_{i},b_{i}]}(s) \, \upchi_{(h_j,l_{j}]}(x),\label{e.2.3}
\end{equation}
where $n$ and $m$ are finite positive integers, $-\infty<a_{1}<b_{1}<\cdots<a_{n}<b_{n}<\infty$, $h_j=(h_{j1}, \cdots, h_{jd})$    $h_{jk}<l_{jk}\ $, $\upchi_{(h_j,l_{j}]}(x)=\prod_{k=1}^d \upchi_{(h_{jk},l_{jk}]}(x_k)$,     and $X_{i,j}$ are ${\mathcal{F}} _{a_{i}}$-measurable random variables for $i=1,\ldots,n$. The integral of  such a process with respect to $W$ is defined as
\begin{eqnarray}
\int_{\RR_+\times \mathbb{R}^d}g(s,x) \, W(ds,dx)
&=&\sum_{i=1}^{n} \sum_{j=1}^m X_{i,j} \, W\lp \upchi_{(a_{i},b_{i}]} \otimes \upchi_{(h_j,l_{j}]}\rp   \label{eq:riemann-sums-W}\\
&=&\sum_{i=1}^{n} \sum_{j=1}^m X_{i,j} \,\big[W(b_{i},l_{j})-W(a_{i},l_{j}) -W(b_i,h_j)+ W(a_i, h_j)\big]\,. \nonumber
\end{eqnarray}

\end{definition}

The following result can be found in \cite{HHLNT1}  when $d=1$.  

\begin{proposition}\label{prop:intg-wrt-W}
Let $\laa_{{H}}$ be the space of predictable processes $g$ defined on $\R_{+}\times\R^d$ such that
almost surely $g\in \HH$ and $\EE[\|g\|_{\HH}^{2}]<\infty$. Then, 
we have the following statements.  

\noindent
\emph{(i)}
The space of elementary processes of the form  in Definition \ref{def:elementary-process} is dense in $\laa_{H}$.

\noindent
\emph{(ii)}
For $g\in\laa_{H}$, the stochastic integral $\int_{\RR_+\times \mathbb{R}^d}g(s,x) \, W(ds,dx)$ is defined as the $L^{2}(\oom)$-limit of $\int_{\RR_+\times \RR^d} g_n(s,x)W(ds,dx)$ for any $g_n$  approximating $g$, and we have
\begin{equation}\label{int isometry}
\EE\lc \lp \int_{\RR_+\times \mathbb{R}^d}g(s,x) \, W(ds,dx) \rp^{2} \rc
=
\EE \left[ \|g\|_{\HH}^{2}\right].
\end{equation}
\end{proposition}
We can also define the multiple integral using the above definition. 
\begin{eqnarray}
I_n(f)(t)
&=&  \int_{0\le s_1<\cdots<s_n\le t}\int_{\RR^{dn}}
f((s_1,x_1),  \cdots, (s_n, x_n))W(ds_1, dx_1)\cdots W(ds_n, dx_n)    \nonumber\\
 &=&  \frac1{n!}  \int_{([0, t]\times \RR^{d })^n}
f((s_1,x_1),  \cdots, (s_n, x_n))W(ds_1, dx_1)\cdots W(ds_n, dx_n)
 \,,\nonumber\\
\end{eqnarray} 
where $f((s_1,x_1),  \cdots, (s_n, x_n))\in \HH^{\otimes n}$ is symmetric with respect to its $n$ arguments.  We have
\begin{eqnarray}
\EE\left[ \left(I_n(f) (t)\right)^2 \right]
&=&  \frac{1}{n!} \int_{ [0,t]^{2n} }\int_{\RR^{dn}} 
\widehat f((s_1,\xi_1),  \cdots, (s_n, \xi_n)) 
\overline{\widehat f}  ((r_1,\xi_1),  \cdots, (r_n, \xi_n))  \nonumber\\
&&\qquad \prod_{i=1}^n\ga_0(s_i-r_i)  \prod_{i=1}^n \prod_{k=1}^d |\xi_{ik}|^{1-2H_k}  d\xi_1\cdots d\xi_n  ds_1\cdots ds_n dr_1 \cdots dr_n \,, \nonumber\\ 
&=&  \int_{{0\le s_1<\cdots<s_n\le t\atop
0\le r_1,\cdots, r_n\le t}}\int_{\RR^{dn}} 
\widehat f((s_1,\xi_1),  \cdots, (s_n, \xi_n)) 
\overline{\widehat f}  ((r_1,\xi_1),  \cdots, (r_n, \xi_n))  \nonumber\\
&&\qquad  \prod_{i=1}^n\ga_0(s_i-r_i)  \prod_{i=1}^n \prod_{k=1}^d |\xi_{ik}|^{1-2H_k}  d\xi_1\cdots d\xi_n dsdr   \,,  \label{e.2.8} 
\end{eqnarray} 
where $\widehat f$ is the Fourier transform with respect to $n$ spatial variables $x_1, \cdots, x_n$; $\xi_i=(\xi_{i1}, \cdots, \xi_{id})$; $d\xi_i=d\xi_{i1}\cdots d\xi_{id}$,
$ds=ds_1\cdots ds_n$, $dr=dr_1\cdots dr_n$.   Notice that in \eqref{e.2.8} we do not force an order for $r_1, \cdots, r_n$. 

We  also need the following lemma which can be found in \cite[Lemma 4.5]{HHNT}. 
\begin{lemma}\label{lemsimplex}
Let $\alpha \in (-1+\ep , 1)^m$  with $\ep>0$. Denote
 $|\alpha |= \sum_{i=1}^m
\alpha_i  $  and 
$T_m(t)=\{(r_1,r_2,\dots,r_m) \in \R^m: 0<r_1  <\cdots < r_m < t\}$.
Then there is a constant $\kappa$, depending only on $\ep$  such that
\[
J_m(t, \alpha):=\int_{T_m(t)}\prod_{i=1}^m (r_{\si(i)}-r_{i-1})^{\alpha_i}
dr \le \frac { \kappa^m t^{|\alpha|+m } }{ \Gamma(|\alpha|+m +1)},
\]
where by convention, $r_0 =0$.
\end{lemma}

\section{Necessary Condition}\label{S:3} 

In this section we shall prove part (ii), the necessary part,  of Theorems \ref{t.1.1} and \ref{t.1.2}, namely, the necessity of
\eqref{e:1.4} and \eqref{e.1.9}, respectively. 

First, we give the meaning of the 
  (random field) solution to equation \eref{spde} in the following  definition.

\begin{definition}\label{def-sol-sigma}
  A  real-valued predictable stochastic process $u=\{u(t,x), 0 \leq t <\infty, x \in \mathbb{R}^d
\}$ is said to be a  {\it    (global) random field  solution}  of \eref{spde} if
\begin{enumerate}
\item[(i)]  for all $t\in[0, \infty)$ and $x\in\R^d$,  the process 
$ (s, y) \mapsto p_{t-s}(x-y) u(s,y)  \upchi_{[0,t]}(s)$ 
is an element of $\laa_{H}$, where $p_t(x)=(2\pi t)^{-d/2}
\exp\left( -\frac{|x|^2}{2t}\right)$ is the heat kernel on the real line  associated with 
 $\frac{1}{2}\Delta$. 
 
 \item[(ii)] for all $t \in [0,\infty)$ and $x\in \mathbb{R}^d$ we have
\begin{equation}\label{e.3.1}
u(t,x)= p_t*u_0(x) + \int_0^t \int_{\mathbb{R}^d}p_{t-s}(x-y) u(s,y)  W(ds,dy) \quad a.s.,
\end{equation}
where the stochastic integral is understood in the sense of Proposition \ref{prop:intg-wrt-W}.
\end{enumerate}
 
\smallskip

 A  real-valued   stochastic process $ u(t,x)$ 
 is said to be a  {\it  local (random field)  solution } of \eref{spde} if there is some constant $t_0>0$
 so that $u(t, x)$ is defined on $[0, t_0) \times \R^d$ and satisfies all the above property 
 for $(t, x)\in [0, t_0) \times \R^d$.

 \smallskip
 
 We say a   random field  solution $u(t, x)$ of \eref{spde}  is in $L^2(\Omega, \mathcal{F},  \PP)$
if $\EE [ u(t, x)^2] <\infty$ for every   $t\geq 0$ and $x\in \R^d$ (for local solution, 
replace $t\geq 0$ by $t\in [0, t_0)$).
\end{definition}

Repeatedly using  this definition, we  see that the solution of \eref{spde} has the following {  Wiener } chaos expansion
\begin{equation}\label{e.3.2} 
u(t,x)=\sum_{n=0}^\infty 
u_n(t,x),
\end{equation}
where 
\begin{equation} \label{e:3.3} 
u_n(t,x)=I_n(f_n^{(t,x)}) (t)
\end{equation}
with 
\begin{eqnarray} 
f_n^{(t,x)} 
&:=& f_n^{(t,x)}( s_1,x_1,\dots,s_n,x_n )= f_n^{(t,x)}  ((s_1,x_1),\dots, (s_n,x_n) )\nonumber \\
&=&  p_{t-s_n}(x-x_n)p_{s_n-s_{n-1}}(x_ n-x_{n-1})\cdots p_{s_2-s_1}(x_ 2-x_1)
p_{s_1}u_0(x_1)\,,  \label{e.3.3}
\end{eqnarray} 
 (see, for instance,  formula (4.4) in \cite{HN} or  formula (3.3) in \cite{HHNT}).
 The remainder in \eqref{e.3.2} will go  to zero since  we can expand it first  for finitely
  many terms to deduce
$ \EE [ u(t, x)^2] \geq \sum_{n=0}^\infty \EE \left[ I_2(f_n^{ (t, x)})(t)^2 \right]$.  
 On the other hand, $u(t, x)$ defined by \eqref{e.3.2} is a solution as long as 
  it converges in $L^2$. 

By comparison, without loss of generality we assume $u_0(x)=1$ throughout the remaining part of this paper.  
The Fourier transform of $f_n$ (with respect to the $n$ ($d$ dimensional) spatial variables  
is
\begin{equation}
\widehat f_n(t,x, s_1, \xi_1, \cdots, s_n, \xi_n)
=\prod_{i=1}^n  e^{-\frac{1}{2} (s_{i+1}-s_i)|\xi_i+\cdots+\xi_1|^2} e^{-{  i}  x(\xi_n+\cdots+\xi_1)}\,, 
\label{e.3.4} 
\end{equation}
when  $0<s_1<\cdots<s_n<s_{n+1}=t$, where we denote $s_{n+1}=t$  (see \cite[3.13-3.14]{hule2019} or \cite[p.8]{HHLNT2}).  In general, we have
\begin{equation}
\widehat f_n(t,x, s_1, \xi_1, \cdots, s_n, \xi_n)
=\prod_{i=1}^n  e^{-\frac{1}{2} (s_{\si(i+1)}-s_{\si(i)})
|\xi_{\si(i)}+\cdots+\xi_{\si(1)}|^2} e^{-  i  x(\xi_n+\cdots+\xi_1)}\,,
\label{e.3.4a} 
\end{equation}
where $0<s_{\si(1)} <\cdots <s_{\si(n)}<s_{\si(n+1)}=t$.  
 
 \medskip 
  
 We will need the following  elementary  lemma.

\begin{lemma}\label{l.3.2}  Let $\vare>0$ and let $0<\al, \be<1$ with $\al+\be 
 >1$. 
 There is a constant $c\geq 1$ independent of $\vare>0$ so that for all $x\in (0, 3\vare )$, 
\begin{equation}\label{e:3.6}
c^{-1}   x^{1-(\al+\be)}  \leq \int_0^\vare u^{-\al} (u+x)^{-\be} du\leq c  x^{1-(\al+\be)}   .
\end{equation}
\end{lemma}
 
\begin{proof} By a change of variable $u=xv$, 
$$
\int_0^\vare u^{-\al} (u+x)^{-\be} du = x^{1-(\al + \be)} \int_0^{\vare/x}   v^{-\al} (1+v)^{-\be} dv .
$$
The   desired conclusion \eqref{e:3.6} follows from this.   
  \end{proof} 
 \medskip

In the following, we use $C_H$ to denote a positive constant depending on $H=(H_1, \cdots, H_d)$ as well as the dimension 
 $d\geq 1$, whose exact value is unimportant and may change from line to line. 
   For two non-negative functions $f$ and $g$,
  notation $f\asymp g$ means that there is a constant $c\geq 1$ so that $c^{-1} f \leq g\leq c f $
  on a specified common definition domain of $f$ and $g$. 
  
\medskip 
\noindent {\it Proof of Part  (ii)  of Theorems
\ref{t.1.1} and  \ref{t.1.2}}.\ \ 
First,  we consider
 the one dimensional case $d=1$ with $H_0>1/2$.  Denote $H=H_1$.

 Let us consider  the second chaos in 
\eqref{e.3.2}.  From now on we denote  
$ I_n(f_2^{t,0}) = I_n(f_2^{t,0})(t)$. 
\begin{eqnarray}
\EE \left[   I_2(f_2^{t,0}) ^2 \right] 
&=& \frac12    \int_{{0\le s_1  < s_2\le t\atop
 0\le   r_1< r_2\le t}   }\int_{\RR^2}
 e^{-  \frac12 (t-s_2+t-r_1  )|\xi_2+ \xi_1|^2-\frac12 (s_2-s_1+r_2-r_1)|\xi_1|^2
 }\nonumber\\
 &&\qquad 
 |\xi_1|^{1-2H}|\xi_2|^{1-2H}d\xi_1 d\xi_2 \ga_0(s_2-r_2) \ga_0(s_1-r_1) ds_1ds_2
 dr_1dr_2\nonumber\\
 &&\quad + \frac12    \int_{{0\le s_1  < s_2\le t\atop
 0\le   r_2< r_1\le t}   }\int_{\RR^2}
 e^{-  \frac12 (t-s_2+t-r_2)|\xi_2+ \xi_1|^2- \frac{s_2-s_1 }{2} |\xi_1|^2-\frac{r_1-r_2}{2}|\xi_2|^2}\nonumber\\
 &&\qquad 
 |\xi_1|^{1-2H}|\xi_2|^{1-2H}d\xi_1 d\xi_2 \ga_0(s_2-r_2) \ga_0(s_1-r_1) ds_1ds_2
 dr_1dr_2\nonumber\\
 &\ge & \frac12 \int_{{0\le s_1\le s_2\le t\atop 0\le r_1<r_2\le t}}g(s_1, s_2, r_1, r_2 ) \ga_0(s_2-r_2) \ga_0(s_1-r_1)  dr_1dr_2 ds_1ds_2 \,,\nonumber\\
 \label{e.3.6} 
\end{eqnarray}
where 
\[  
g(s_1, s_2, r_1, r_2 )=\int_{\RR^2}
 e^{-  \frac12 (t-s_2+t-r_1  )|\xi_2+ \xi_1|^2-\frac12 (s_2-s_1+r_2-r_1)|\xi_1|^2
 }
 |\xi_1|^{1-2H}|\xi_2|^{1-2H}d\xi_1 d\xi_2  . 
\]
Making substitution $\eta_1=\xi_1$ and $\eta_2=\xi_1+\xi_2$, we have for $0<s_1<s_2<t$
and $0<r_1<r_2<t$, 
\begin{eqnarray}
g(s_1, s_2, r_1, r_2 )
&=&\int_{\RR^2}
 e^{- \frac12  (t-s_2+t-r_2)|\eta_2|^2-\frac12 (s_2-s_1+r_2-r_1)|\eta_1|^2} 
 |\eta_1|^{1-2H}|\eta_2-\eta_1|^{1-2H}d\eta_1 d\eta_2\nonumber\\
&=&\frac{C_H}{\sqrt{(t-s_2+t-r_2)(s_2-s_1+s_2-s_1)}}   \EE \Bigg[ 
\left|\frac{X_1 }{\sqrt{s_2-s_1+r_2-r_1}}
\right|^{1-2H}\nonumber\\
&&\qquad\quad  \times \left|\frac{X_2}{\sqrt{t-s_2+t-r_2}}-\frac{X_1 }{\sqrt{s_2-s_1+r_2-r_1}}\right|^{1-2H}\Bigg]\,,\nonumber\\
&= & C_H  (t-s_2+t-r_2)^{H-1 }(s_2-s_1+r_2-r_1)^{2H-3/2}  \nonumber\\
&& \quad  \times \EE \left[ 
\left| X_1 ( \sqrt{s_2-s_1+r_2-r_1} X_2 -\sqrt{t-s_2+t-r_2} X_1)\right|^{1-2H}\right] \nonumber\\
&=& C_H  (t-s_2+t-r_2)^{H-1 }(s_2-s_1+r_2-r_1)^{2H-3/2} (t-s_1+t-r_1)^{1/2-H} \nonumber\\
&&\qquad\quad  \times \EE \left[ 
\left| X_1 \left( \frac{\sqrt{s_2-s_1+r_2-r_1}}{\sqrt{t-s_1+t-r_1}} X_2 -\frac{\sqrt{t-s_2
+t-r_2}}{\sqrt{t-s_1+t-r_1}} X_1\right)\right|^{1-2H}\right] ,  \nonumber \\
\label{e:3.8}
 \end{eqnarray}
  where $X_1$ and $X_2$ are two independent standard Gaussian  
random variables.   Denote
\[
f(\lambda):= \EE \left[ |X_1 (\lambda X_1-\sqrt{1-\lambda^2}   X_2)|^{1-2H}\right]
\,,\quad \lambda \in [0, 1]\,.
\]
We claim that   
\begin{equation} \label{e:3.9}
 \min_{\lambda \in [0, 1]} f(\lambda )>0 \quad \hbox{for any } 0 < H <1.
 \end{equation} 
 First   note that for a standard Gaussian random variable $Z$ and $a\in \R$, 
by  the Hardy-Littlewood (symmetric rearrangement) inequality, 
\begin{eqnarray} \label{e:3.10} 
\EE \left[ |Z-a|^{1-2H} \right]&=&\frac{1}{\sqrt{2\pi}} \int_{-\infty}^\infty |z-a|^{1-2H} e^{-|z|^2/2} dz  \nonumber \\
&\leq& \frac{1}{\sqrt{2\pi}} \int_{-\infty}^\infty |z |^{1-2H} e^{-|z|^2/2} dz  = \EE \left[ |Z|^{1-2H} \right].
\end{eqnarray}
Taking conditional expectation on $\sigma (X_1)$ and then 
using \eqref{e:3.10}, we have for any $0<H<1$, 
\begin{eqnarray*}
f(\lambda) &=& \EE \left[ |X_1|^{1-2H} \, \EE \left(  |\lambda X_1-\sqrt{1-\lambda^2}   X_2)|^{1-2H} \big| \sigma (X_1) \right) \right] \\
&\leq &    (1-\lambda^2)^{1/2 -H}  \, \EE \left[ |X_1|^{1-2H} \right] \, \EE \left[ |X_2|^{1-2H} \right] <\infty. 
\end{eqnarray*}
We conclude by the dominated convergence theorem that $f(\lambda)$ is a positive continuous function on $[0, 1)$ with
$f(0)=   \left(  \EE \left[ |X_1   |^{1-2H}\right] \right)^2 \in (0, \infty)$. 
{We need to consider  the behavior of $f (\lambda) $ near $\lambda =1$. } 

When $0<H<3/4$,  {$f$ is a continuous function on $[0, 1]$ with    $f(1)= \EE \left[ |X_1|^{2-4H}\right]
  \in (0, \infty)$ and so $f$ is bounded between two positive constants on $[0, 1]$. In particular,  we have \eqref{e:3.9}.} 
When $3/4\le H<1$, by  Fatou's lemma,  
 $$ \liminf_{\lambda \to 1-} f(\lambda) { \geq f(1) }
 = \EE \left[ |X_1|^{2(1-2H)} \right]  =\infty.
 $$ 
This establishes the claim \eqref{e:3.9}
for all $H\in [0, 1]$. 
Consequently, we have by \eqref{e:3.8} that for any $H\in (0, 1)$, 
there is a constant $C_H>0$ so that 
\begin{equation} \label{e:3.11}
g(s_1, s_2, r_1, r_2 ) \geq  C_H (t-s_1+t-r_1)^{1/2-H}  (t-s_2+t-r_2)^{H-1 }(s_2-s_1+r_2-r_1)^{2H-3/2} \,. 
\end{equation}   
 In order for  $\EE \left[   I_2(f_2^{t,0})(t)^2 \right]$ to be finite,  
  by  \eqref{e.3.6} and \eqref{e:3.11} 
  the following integral must be finite:
\begin{eqnarray} \label{e:3.13}
&&
\Upsilon:=\int_{{0\le s_1<s_2\le t\atop 0\le r_1<r_2\le t}} 
(t-s_1+t-r_1)^{1/2-H}  (t-s_2+t-r_2)^{H-1 }(s_2-s_1+r_2-r_1)^{2H-3/2}\nonumber\\
&&\qquad\qquad  \times   |s_2-r_2|^{2H_0-2}|s_1-r_1|^{2H_0-2}    dr_1dr_2 ds_1ds_2 \,. 
\end{eqnarray} 
  It is obvious  that
  { 
\begin{eqnarray} 
\Upsilon&\ge & \int_{    0\le r_1<s_1<r_2<s_2\le t } 
(t-s_1+t-r_1)^{1/2-H}  (t-s_2+t-r_2)^{H-1 }(s_2-s_1+r_2-r_1)^{2H-3/2}\nonumber\\
&&\qquad\qquad   \times  |s_2-r_2|^{2H_0-2}|s_1-r_1|^{2H_0-2}    dr_1dr_2 ds_1ds_2 \nonumber\\
&\ge& \int_{    0\le r_1<s_1<r_2<s_2\le t } 
 (t-s_2+t-r_2)^{ -1 /2}(s_2-s_1+r_2-r_1)^{2H-3/2}\nonumber\\
&&\qquad\qquad   \times |s_2-r_2|^{2H_0-2}|s_1-r_1|^{2H_0-2}    dr_1dr_2 ds_1ds_2 \,. 
\label{e.3.8} 
\end{eqnarray} 
Making substitution from $r_1, s_1, r_2$ to $u=s_1-r_1$, $v=r_2-s_1$, $w=s_2-r_2$,
 we have
\begin{eqnarray} 
\Upsilon 
&\ge& \int_{  u,v,w>0,  u+v+w< s_2<t } 
 (2t-2s_2+w)^{ -1 /2}(u+2v+w)^{2H-3/2}
   w^{2H_0-2}u^{2H_0-2}    dudvdwds_2 \nonumber\\
 &\ge &  \int_{    v,w>0,  \,  v+w\le  s_2 /2,  \, s_2<t }  \left(\int_0^{s_2/2}  (u+2v+w)^{2H-3/2} u^{2H_0-2} du \right)
\nonumber\\ &&\qquad\qquad   \times 
   (2t-2s_2+1)^{ -1 /2} w^{2H_0-2} dw dv ds_2
 \nonumber \\
&\ge &  c_1\int_{    v,w>0, \,  v+w\le s_2/2, \, s_2<t } 
 (2t-2s_2+1)^{ -1 /2}(2v+w)^{2H_0 +2H-5/2}
    w^{2H_0-2}    dvdwds_2 \nonumber\\
    &\ge & c_1\int_{    0<  v \le s_2/4, \, s_2<t } 
 (2t-2s_2+1)^{ -1 /2} \left( \int_0^{s_2/4} (2v+w)^{2H_0 +2H-5/2}
    w^{2H_0-2}    dw \right) dvds_2 \nonumber\\
&\ge  &  c_1^2 \int_{   0<  v \le s_2/4, \, s_2<t }  
 (2t-2s_2+1)^{ -1 /2}  v  ^{4H_0 +2H-7/2}    dv ds_2 , \label{e.3.11} 
\end{eqnarray} 
where in the third and fifth inequality we used Lemma  \ref{l.3.2}.
So $\Upsilon <\infty \,$  implies  the integral in \eqref{e.3.11} is finite, which happens  }
  only when $4H_0 +2H-7/2>-1$, that is,  
\begin{equation}
4H_0 +2H>5/2\,. 
\end{equation}   
This proves 
  Theorem \ref{t.1.2}(ii) when $d=1$.  

When $H_0=1/2$, the equation \eqref{e:3.13}    becomes
\begin{eqnarray} \label{e:3.17a}
\Upsilon= { 2^{2H-2} } 
 \int_{    0  <s_1 <s_2\le t } 
(t-s_1)^{1/2-H}  (t-s_2 )^{H-1 }(s_2-s_1 )^{2H-3/2} ds_1ds_2 \,. 
\end{eqnarray} 
This integral is finite only when the exponent $2H-3/2$ of $s_2-s_1$ 
in \eqref{e:3.17a} is  larger than $-1$. This requires  $H>1/4$.  Notice  that by formally letting $H_0=1/2$ in
$4H_0 +2H>5/2$ one also obtains  $H>1/4$. 
This proves  
 Theorem \ref{t.1.1}(ii) when $d=1$.   

\medskip

Next we  consider the case that the dimension $d\ge 2$. We still consider the second chaos. As for the one dimensional case we have
  \begin{eqnarray}
\EE \left[  (I_2(f_2))^2 \right] 
&\ge & \frac12 \int_{{0\le s_1\le s_2\le t\atop 
0\le r_1\le r_2\le t} }\int_{\RR^{2d} }
 e^{-  \frac12(t-s_2+t-r_2)|\xi_2+ \xi_1|^2-\frac12 (s_2-s_1+r_2-r_1 )|\xi_1|^2}\nonumber\\
 &&\qquad 
  \times \prod_{k=1}^d |\xi_{1k}|^{1-2H_k}|\xi_{2k} |^{1-2H_k}  \ga_0(s_1-r_1)\ga_0(s_2-r_2) d\xi_1 d\xi_2 ds_1ds_2 dr_dr_2\nonumber\\
 &=& \int_{{0\le s_1\le s_2\le t\atop 
0\le r_1\le r_2\le t} }\prod_{k=1}^d g_k(s_1, s_2, r_1, r_2)  \ga_0(s_1-r_1)\ga_0(s_2-r_2) ds_1ds_2 dr_1dr_2\,,  \nonumber\\
 \label{e.3.13} 
\end{eqnarray}
where 
 \begin{equation}\label{e.3.14} 
g_k(s_1, s_2, r_1, r_2) = 
\int_{\RR^2}
 e^{-  \frac12(t-s_2+t-r_2)|\eta_2|^2 -\frac12 (s_2-s_1+r_2-r_1)|\eta_1|^2} 
 |\eta_1|^{1-2H_k}|\eta_2-\eta_1|^{1-2H_k}d\eta_1 d\eta_2 , 
 \end{equation} 
which can be estimated by using   \eqref{e:3.11}. Thus, we have 
\begin{eqnarray}
\EE \left[ (I_2(f_2))^2 \right] 
&\ge & C \int_{{0\le s_1\le s_2\le t\atop 
0\le r_1\le r_2\le t} } (t-s_1+t-r_1)^{d/2-{H^*} }  (t-s_2+t-r_2)^{{H^*} -d }\nonumber\\
&&  \times (s_2-s_1+r_2-r_1)^{2{H^*} -3d/2}  |s_1-r_1|^{2H_0-2} |s_2-r_2|^{2H_0-2}  ds_1ds_2 dr_1dr_2 \nonumber \\
 \label{e:3.16}
 \end{eqnarray}
With the same argument as for \eqref{e.3.11}, we see that the above integral is finite  only if
\[
2{H^*} -\frac{3d}{2}+4H_0-2>-1\,.
\]
This proves part (ii), the necessary part, of Theorem \ref{t.1.2} for  $d\geq 2$.

\smallskip

When $H_0=1/2$,   the  inequality \eqref{e:3.16}
becomes 
\begin{eqnarray}
\EE \left[ (I_2(f_2))^2 \right] 
&\ge & C \int_{ 0\le s_1\le s_2\le t } (t-s_1 )^{d/2-{H^*} }  (t-s_2 )^{{H^*} -d } (s_2-s_1 )^{2{H^*} -3d/2}     ds_1ds_2  \nonumber \\
&=& C\int_0^t (t-s_1)^{d/2-{H^*} }  \left[\int_{s_1}^{t}  (t-s_2)^{{H^*} -d }(s_2-s_1)^{2{H^*} -3d/2}  ds_2 \right]  ds_1    \nonumber \\
 &=& CB({H^*} -d+1,2{H^*} -\frac{3d}{2}+1) \int_0^t (t-s_1)^{-d/2} (t-s_1)^{2{H^*} -\frac{3d}{2}+1} ds_1 
 \nonumber \\
&=&  CB({H^*} -d+1,2{H^*} -\frac{3d}{2}+1) \int_0^t (t-s_1)^{2{H^*} -2d+1} ds_1, 
\label{e:3.17}
\end{eqnarray}
where $B$ is the beta function  and where we see that 
  $\EE \left[ (I_2(f_2))^2 \right]<\infty$ if and only if  when  $2{H^*} -3d/2>-1$ and ${H^*} -d>-1$.
Note that for $d\ge 2$, ${H^*} -d>-1$ implies that
  $2{H^*} -3d/2+1=2({H^*} -d)+(d/2)+1>0$.
  This completes the  proof of    part (ii), the necessary part, of Theorem \ref{t.1.1} for $d\ge 2$.

\section{Sufficient Condition}\label{S:4} 

\noindent In this section,  we prove part (i), the sufficient part,  of Theorems \ref{t.1.1} and \ref{t.1.2}.  
   It suffices to consider the case that $d\geq 2$, as when $d=1$ conditions \eqref{e:1.4} and \eqref{e:1.7} 
  coincide with  \eqref{e:1.3} so the result for $d=1$ follows from  \cite[Theorems 1.2 and  1.3]{chen}.  

 Recall that we take $u_0(x)=1$ on $\R^d$.  Let $u_n(t, x)$ be defined as in \eqref{e:3.3}.  
   We compute the $L^2(\Omega, \mathcal{F}, \PP)$ norm of each $u_n(t, x)$.   
For $n\geq 0$,    we have by \eqref{e.2.8} 
that  for $0<s_1<\cdots<s_n<t$ and $0<r_{\si(1)}<\cdots <r_{\si(n)}<t$, 
\begin{eqnarray}
&& \EE \left[ u_n^2(t,x) \right]  \nonumber\\
&=& \int_{{0<s_1<\cdots<s_n<t\atop 0< r_1, \cdots, r_n<t}
}\int_{\RR^{nd}}
\prod_{i=1}^n e^{-(s_{i+1}-s_i+r_{\si(i+1)}-r_{\si(i)})|\xi_i+\cdots+\xi_1|^2} 
 \prod_{k=1}^d 
|\xi_{ik}|^{1-2H_k}d\xi_1\cdots d\xi_n \nonumber\\
&&\qquad 
 \times \prod_{i=1}^n \ga_0(s_i-r_i) ds_1dr_1\cdots ds_ndr_n\nonumber\\
&=& \int_{[0, t]^{2n}} \prod_{k=1}^d g_k(s_1, \cdots, s_n, r_1, \cdots, r_n) 
\prod_{i=1}^n \ga_0(s_i-r_i) ds_1dr_1\cdots ds_ndr_n\,,
\label{e.4.1} 
\end{eqnarray}
where   $s_{n+1}=r_{n+1}:=t$ and 
\begin{eqnarray}
&&g_k(s_1, \cdots, s_n, r_1, \cdots, r_n)\nonumber\\
&=& \int_{\RR^{n }}
\prod_{i=1}^n e^{-(s_{i+1}-s_i+r_{\si(i+1)}-r_{\si(i)})|\xi_{ik}+\cdots+\xi_{1k}|^2}
 |\xi_{ik}|^{1-2H_k}d\xi_{1k}\cdots d\xi_{nk}\nonumber\\
&=&   \int_{\RR^{n }}
\prod_{i=1}^n e^{-(s_{i+1}-s_i+r_{\si(i+1)}-r_{\si(i)})|\eta_i|^2}
 |\eta_i-\eta_{i-1}|^{1-2H_k}d\eta_1\cdots d\eta_n\,,
\end{eqnarray}
where   $\eta_0:=0$.   
Denote $u_i=s_{i+1}-s_i+r_{\si(i+1)}-r_{\si(i)}$ for  $1\leq i\leq  n$. Let    $u_0=1$,   $X_0=0$,   
and $\{X_1, \cdots,  X_n\}$ be  i.i.d standard Gaussian random variables.    
Then we can write for $0<s_1<s_2<\cdots < s_{n}<t$, 
\begin{eqnarray}
&&g_k(s_1, \cdots, s_n, r_1, \cdots, r_n)
=  c_{H_k}^n  \left( \prod_{i=1}^n u_i^{-1/2} \right) \EE\left[\prod_{i=1}^n 
\left|\frac{X_i}{\sqrt{u_i}}-\frac{X_{i-1}}{\sqrt{u_{i-1}}}\right|^{1-2H_k}\right]\nonumber\\
& &\qquad = c_{H_k}^n  \left( \prod_{i=1}^n u_i^{-1/2} \right)
\left(   \prod_{i=1}^n    (u_iu_{i-1})^{H_k-1/2} \right) \EE\left[ \left|  X_1 
 \right| ^{1-2H_k}  \prod_{i=2}^n 
\left|  \sqrt{u_{i-1}}  X_i - \sqrt{u_i}  X_{i-1}  \right|^{1-2H_k} \right]
\nonumber\\
 & &\qquad =  c_{H_k}^n u_n^{H_k-1} 
\left( \prod_{i=2}^{n-1} u_i^{2H_k-3/2} \right)
\left( \prod_{  i=2 }^{n }  (u_i+u_{i-1})^{\frac12- H_k} \right)  \nonumber\\
& &\qquad \qquad  \times   \EE\left[ \left|  X_1  \right| ^{1-2H_k} \prod_{i=2}^n 
\left| \sqrt{\frac{u_{i-1}}{u_{i-1}+u_i} }X_i - \sqrt{\frac{u_{i }}{u_{i-1}+u_i} }X_{i-1} \right|^{1-2H_k}
\right] \,. \label{e.4.3} 
\end{eqnarray} 
Denote $\la_i=\sqrt{\frac{u_{i-1}}{u_{i-1}+u_i}}$. The  expectation 
(denoted by $I_{k,n}$) in \eqref{e.4.3} is bounded  as follows.
\begin{eqnarray}
I_{k,n} 
&:=& \EE\left[ \left|  X_1 
 \right| ^{1-2H_k}   \prod_{i=2}^{d_*}  
\left|  {\la_i }X_i - \sqrt{1-\la_i^2 }X_{i-1} \right|^{1-2H_k} 
\cdot \prod_{i=d_*+1}^{n}  
\left|  {\la_i }X_i - \sqrt{1-\la_i^2 }X_{i-1} \right|^{1-2H_k}   \right] \nonumber\\
 &\le & C_{d_*, H_k}\EE\left[ \left|  X_1 
 \right| ^{1-2H_k} \left( \prod_{i=2}^{d_*}  \left(
 |  X_i | \vee |X_{i-1} \right)^{1-2H_k} \right)
\left( \prod_{i=d_*+1}^{n}  
\left|  {\la_i }X_i - \sqrt{1-\la_i^2 }X_{i-1} \right|^{1-2H_k}  \right) \right] \nonumber\\
&\le &  C_{d_*, H_k} \left(  \prod_{i=1}^{d_*-1} 
\left( \EE  \left[\left|  X_i   \right| ^{1-2H_k}\right] \right)  \vee \EE  \left[\left|  X_i   \right| ^{2-4H_k}\right] \right) 
 \nonumber \\
 &&  \quad \times \EE\left[ \left|  X_{d_*}  
 \right| ^{1-2H_k}  \prod_{i=d_*+1}^{n}  
\left|  {\la_i }X_i - \sqrt{1-\la_i^2 }X_{i-1} \right|^{1-2H_k} \right] \nonumber\\
&=&  C_{d_*, H_k}   \EE\left[ \left|  X_{d_*}  \right| ^{1-2H_k}  \prod_{i=d_*+1}^{n}  
\left|  {\la_i }X_i - \sqrt{1-\la_i^2 }X_{i-1} \right|^{1-2H_k} \right]\, ,    \label{e:4.4}
 \end{eqnarray}  
 with the convention that $\prod_{i= m}^n a_i :=1$ for $m>n$. 
 
 To bound the remaining expectation we use 
 {the following estimate for standard normal random variable $X$
 from \cite[Lemma A.1]{HNS}: there is a constant $C>0$ so that for any $0<\alpha <1$, $ \lambda >0$ and $b>0$,
 \begin{equation}\label{e:4.5} 
 \EE \left[ | \lambda X +b |^{-\alpha} \right] \leq C  (\lambda \vee b )^{-\alpha} . 
 \end{equation}   
    By taking conditional expectation on the $\sigma$-field  $\sigma  ( X_{d_*}, \cdots X_{n-1})$ and using 
    (\eqref{e:4.5},  } 
\begin{eqnarray*}
&& \EE\left[ \left|  X_{d_*}   \right| ^{1-2H_k}  \prod_{i=d_*+1}^{n}  
\big|  {\la_i }X_i - \sqrt{1-\la_i^2 } \, X_{i-1} \big|^{1-2H_k} \right]\\
& = &\EE\left[ \EE \Big[ \left|  X_{d_*}   \right| ^{1-2H_k}  \prod_{i=d_*+1}^{n}  
\big|  {\la_i }X_i - \sqrt{1-\la_i^2 } \, X_{i-1} \big|^{1-2H_k} \Big| \sigma  ( X_{d_*}, \cdots X_{n-1}) \Big] 
\right]\\ \\
& \leq  & C_{H_k}  \EE\left[ \left|  X_{d_*}  \right| ^{1-2H_k}
 \prod_{i=d_*+1}^{n-1}  \big|  {\la_i }X_i - \sqrt{1-\la_i^2 } \, X_{i-1} \big|^{1-2H_k} \,
    \left(   \Big (\sqrt{1-\la_n^2 } \, |X_{n-1}| \Big) \vee \la_n \right)^{1-2H_k}    
  \right]\\
& \le & C_{H_k} \la_n^{1-2H_k} \EE\left[  \left|  X_{d_*}  
 \right| ^{1-2H_k}    \prod_{i=d_*+1}^{n-1}  
\big|  {\la_i }X_i - \sqrt{1-\la_i^2 }X_{i-1} \big|^{1-2H_k} \right] \\
&  \le & \cdots \le \, C_{H_k}^n 
   \prod_{i=d_*+1}^n  \la_i^{1-2H_k}  \, .
  \end{eqnarray*}
Thus we have by \eqref{e.3.4}-\eqref{e:4.4} that  
\begin{eqnarray}
&&g_k(s_1, \cdots, s_n, r_1, \cdots, r_n)\nonumber\\
&\le & c_{H_k}^n u_n^{H_k-1} \left( \prod_{i=1}^{n-1} u_i^{2H_k-3/2} \right) \left( \prod_{i=2}^n
(u_i+u_{i-1})^{\frac12- H_k} \right)  \prod_{i=d_*+1}^n
\la_i^{1-2H_k}    \nonumber\\
&=&  c_{H_k}^n u_n^{H_k-1} \left( \prod_{i=1}^{n-1} u_i^{2H_k-3/2} \right)
  \left( \prod_{i=2}^{d_*}  (u_i+u_{i-1})^{\frac 12-H_k} \right)   \prod_{i=d_*+1}^n
u_{i-1}^{\frac12 -H_k}   \, .   \nonumber \\
 \end{eqnarray}
  Consequently,  
\begin{eqnarray}
{ 
g (s_1, \cdots, s_n, r_1, \cdots, r_n)}  &=&  \prod_{k=1}^d {  g_k(s_1, \cdots, s_n, r_1, \cdots, r_n)  } 
\nonumber    \\
&\le &  c_{H }^n u_n^{H -d} \left( \prod_{i=1}^{n-1} u_i^{2H -\frac{3d}{2}} \right) 
 \left( \prod_{i=2}^{d_*}  (u_i+u_{i-1})^{\frac d2- {H^*}  } \right) 
 \prod_{i=d_*+1}^n u_{i-1}^{\frac d2 -{H^*}  } \,. \nonumber \\
 \label{e.4.6} 
\end{eqnarray} 

{(i) We first consider the case $1/2<H_0<1$. } 
  When $d\ge 2$ and ${H^*} >d-1$, we  have  ${H^*} > d/2$.  We bound 
$(u_i+u_{i-1})^{\frac d2- {H^*}   }$ in  
 \eqref{e.4.6} by $ u_{i-1} ^{\frac d2- {H^*}   }$. Therefore 
\begin{eqnarray}
g (s_1, \cdots, s_n, r_1, \cdots, r_n) 
&\le &  c_{H }^n u_n^{{H^*}  -d} \left( \prod_{i=1}^{n-1} u_i^{2{H^*}  -\frac{3d}{2}} \right) 
\left( \prod_{i=2}^{d_*}
  u_{i-1} ^{\frac d2- {H^*}  } \right)  \prod_{i=d_*+1}^n
u_{i-1}^{\frac d2 -{H^*}  }\nonumber\\
&  =&
c_{H }^n    \prod_{i=1}^{n} u_i^{ {H^*}  -d }  \,. \label{e.4.7} 
\end{eqnarray} 
It follows then 
\begin{eqnarray}
\EE \left[ u_n^2(t,x) \right] 
&=& 
\int_{{0< s_1<\cdots<s_n< t\atop 0<r_1, 
\cdots, r_n<t} }   g (s_1, \cdots, s_n, r_1, \cdots, r_n)
\prod_{i=1}^n \ga_0(s_i-r_i)  ds_1dr_1\cdots ds_ndr_n\nonumber\\
&\le & C_H^n 
\int_{{0< s_1<\cdots<s_n< t\atop 0<r_1, 
\cdots, r_n<t} }  \prod_{i=1}^n  (s_{i+1}-s_i+r_{\si(i+1)}-r_{\si(i)})^{ {H^*}  -d } 
\nonumber\\
&&\qquad  \times \, \prod_{i=1}^n \ga_0(s_i-r_i)  ds_1dr_1\cdots ds_ndr_n \,. 
\label{e.4.8} 
 \end{eqnarray}
 We use $(a+b)^{-\be}\le a^{-\be/2}b^{-\be/2}$ for all $a, b, \be>0$
 to get 
 \begin{eqnarray}
\EE \left[ u_n^2(t,x) \right] 
&\le & C_H^n 
\int_{{0< s_1<\cdots<s_n< t\atop 0<r_1, 
\cdots, r_n<t} }  \prod_{i=1}^n  (s_{i+1}-s_i )^{ 
\frac{{H^*}  -d }{2} } \prod_{i=1}^n ( r_{\si(i+1)}-r_{\si(i)})^{ 
\frac{{H^*}  -d }{2} } \nonumber\\
&&\qquad \times \, \prod_{i=1}^d \ga_0(s_i-r_i)  ds_1dr_1\cdots ds_ndr_n\nonumber\\
&= & \frac{C_{  H }^n}{ n!} \int_{[0, t]^{2n}}  h(s_1, \cdots, s_n) 
h(r_1, \cdots, r_n) 
 \prod_{i=1}^d \ga_0(s_i-r_i)  ds_1dr_1\cdots ds_ndr_n \,, \nonumber\\
 \end{eqnarray} 
where $h_n(s_1, \cdots, s_n)$ is the symmetric extension
to $[0, t]^n$   of the function $\prod_{i=1}^n  (s_{i+1}-s_i )^{ 
\frac{{H^*}  -d }{2} }  $  defined on $0<s_1<\cdots<s_n<t$.  
Using the multidimensional version of the Hardy-Littlewood
inequality (see e.g. \cite[ (2.4)]{HN})
we have  
\begin{eqnarray}
\EE \left[ u_n^2(t,x) \right]
&\le &      \frac{C_{  H }^n}{ n!}  \left[\int_{[0, t]^n}  h(s_1, \cdots, s_n)^{1/H_0}     ds_1 \cdots ds_n \right]^{2H_0} \, \nonumber\\
&\le &      C_{  H }^n ( n! )^{2H_0-1}
  \left[\int_{0<s_1<\cdots<s_n<t} \prod_{i=1}^n  (s_{i+1}-s_i )^{ 
\frac{{H^*}  -d }{2H_0} }  ds_1 \cdots ds_n \right]^{2H_0} \,. \nonumber\\
 \end{eqnarray} 
When
\begin{equation}
\frac{{H^*} -d}{2H_0}>-1\,, 
\end{equation}
we may use  Lemma \ref{lemsimplex} to bound the above multiple integral
to obtain
\begin{eqnarray}
\EE \left[ u_n^2(t,x) \right] 
&\le &      C_{  H }^n ( n! )^{2H_0-1}
  \left[ \frac{C_{d, H, H_0} ^n} {\Ga\left(\left(\frac{{H^*} -d}{2H_0}+1\right)n+1\right)} t^{ \frac{({H^*} -d)n}{2H_0}+n} \right]^{2H_0} \,.  
  \end{eqnarray} 
For any $p\in [2, \infty)$, by hypercontractivity inequality  $\| u_n (t,x)\|_p \le (p-1)^{n/2}   \| u_n (t,x)\|_2
 $, we have
\begin{eqnarray}
&& \|u_n  (t,x)\|_p  \nonumber\\
 &\le &  p^{n/2}   \left(\EE \left[ u_n^2(t,x) \right] \right)^{p/2}\nonumber\\
 &\le&  C_{  H }^{n /2} p^{n/2}  ( n! )^{(H_0-1/2)  }
  \left[ \frac{C_{d, H, H_0} ^n} {\Ga\left(\left(\frac{{H^*} -d}{2H_0}+1\right)n+1\right)} t^{ \frac{({H^*} -d)n}{2H_0}+n} \right]^{ H_0  } \,.
  \label{e:4.14}
  \end{eqnarray}
  When $H^*>d-1$,  using Stirling's formula for the gamma function that 
 \begin{equation}\label{e:4.15} 
 \Gamma (z) = \sqrt{2\pi/z} \,  (z/e)^z \left(1+ O(1/z)\right)  \quad  \hbox{as } z\to \infty,
 \end{equation} 
  we have by  \eqref{e:4.14}
 $$
 \|u_n  (t,x)\|_p   \leq \frac{C_{d, H, H_0} ^{n } p^{n/2} } {\Ga\left( ({H^*} -d +1)n /2 +1\right)} t^{ ({H^*} -d+2H_0)  n /2}  \,. 
$$
  This implies by the asymptotic behaviour of the  Mittag-Leffler  function  
  (e.g. \cite[p.41, Formula (1.8.10)]{kilbas})  that for all $t>0$
\begin{eqnarray}
 &&\sum_{n=0}^\infty \| u_n  (t,x)\|_p  
  \, \le \,   \sum_{n=0}^\infty\frac{C_{  H, H_0, d }^n p^{n/2} 
  t^{({H^*} -d+2H_0)  n/2}}{\Gamma(n({H^*} -d+1)/2+1)}\nonumber\\
& \le& C\exp\left[ C_H t^{\frac{{H^*} -d+2H_0}{{H^*} -d+1}} p^{\frac{1}{{H^*} -d+1}}
 \right]<\infty \label{e:4.16} .
\end{eqnarray} 
 It follows that $u(t, x):= \sum_{n=0}^\infty u_n  (t,x)$ converges in $L^p(\Omega, \mathcal{F},  \PP)$
for every $p\in [2, \infty)$,
 and $u(t, x)$ is  a global random field solution to \eqref{spde} with $u(0, x)=1$ satisfying \eqref{e.1.8}.

\medskip

When   $H^*=d-1$,     $H^*>d-2H_0$  as $H_0>1/2$  and 
we have from \eqref{e:4.14} by the Stirling's formula \eqref{e:4.15} that 
$$
\|u_n  (t,x)\|_2   \leq  C_{d, H, H_0} ^{n } 2^{n/2}  t^{ ({H^*} -d+2H_0)  n/2}   \exp \left( 
    a_0   (H^*-d+2H_0)n   \right),  
 $$
  where $a_0:= \frac12 ( 1+ \log (2H_0/( 2H_0-1)) >0$.  Clearly there is some positive constant $T_0=T_0 (d, H, H_0) $ so that 
 $\sum_{n=0}^\infty \|u_n  (t,x)\|_2 <\infty$  for any $t\in [0, T_0)$.
It follows that $u(t, x):=\sum_{n=0}^\infty  u(t, x)$ for $(t, x)\in [0, T_0)\times \R^d$ is a 
local random field solution to \eqref{spde}. 
   This   completes the proof of part (i), the existence part, of Theorem  \ref{t.1.2}.

\medskip

(ii) When $H_0=1/2$ and $H^*>d-1$, we replace $\ga_0(s-r)$ in \eqref{e.4.8} by 
$\delta_{\{0\}}(s-r)$. Thus we have 
\begin{eqnarray}
\EE\left[u_n(t,x)^2\right]
&\le&  C_H^n 
\int_{0< s_1<\cdots<s_n< t }  \prod_{i=1}^n  (s_{i+1}-s_i )^{ {H^*}  -d }
ds_1\cdots ds_n  \nonumber\\
&\le&  \frac{C_{H, d} ^n }{\Ga(({H^*} -d+1)n+1)} t^{({H^*} -d+1)n}\,. 
\end{eqnarray} 
By a similar argument to that of \eqref{e:4.15}, we have 
\[
 \sum_{n=0}^\infty \|u_n  (t,x)\|_p
 \le  \sum_{n=0}^\infty\frac{C_{  H,  d }^n  p^{n/2} t^{({H^*} -d+1)  n/2}}{\Gamma(n({H^*} -d+1)/2+1)}
 \le C\exp\left[ C_{H,d}  t p^{\frac{1}{{H^*} -d+1} } \right]<\infty \,. 
 \]
It follows that $u(t, x):= \sum_{n=0}^\infty u_n(t, x)$ converges in   $L^p(\Omega, \mathcal{F},  \PP)$
for every $p\in [2, \infty)$,
and $u(t, x)$ is  a global random field solution to \eqref{spde} with $u (0, x)=1$ on $\R^d$ satisfying  \eqref{e.1.5}. 
   This   completes the proof of part (i), the existence part, of Theorem  \ref{t.1.1}.

\bigskip


\bigskip
 
 
 
\begin{thebibliography}{99}

\bibitem{CHKN}   Chen, L,  Hu, Y.,    Kalbasi, K.  and   Nualart, D.  \, 
Intermittency for the stochastic heat equation driven by a rough time fractional Gaussian noise.
{\it Probab. Theory Related Fields } \textbf{171} (2018) 431-457.
 
 \bibitem{chen}  Chen, X.  \, 
 Parabolic Anderson model with rough or critical Gaussian noise.
{\it Ann. Inst. Henri Poincar\'e Probab. Stat.  \bf 55} (2019),   941-976.   

\bibitem{chen2} Chen, X. \, 
Parabolic Anderson model with a fractional Gaussian noise that is rough in time. 
{\it Ann. Inst. Henri Poincar\'e Probab. Stat.  \bf 56}  (2020),   792-825.


\bibitem{Dal}Dalang, R. \, 
Extending the martingale measure stochastic integral with applications to spatially homogeneous s.p.d.e.'s.
{\it Electron. J. Probab.} \textbf{4} (1999), no. 6, 29 pp.

 

\bibitem{DPZ} Da Prato,G. and    Zabczyk, J. \, 
{\it Stochastic Equations in Infinite Dimensions}. 
Cambridge University Press, 1992.

\bibitem{hu02} Hu, Y.  \, 
Chaos expansion of heat equations with white noise potentials.  
{\it Potential Anal. }  \textbf{16} (2002), no. 1, 45-66.


\bibitem{Hu} Hu, Y. \, 
Some recent progress on stochastic heat equations.
 {\it Acta Math Sci. \bf  39} (2019),
874-914. 

  
\bibitem{HHLNT1}Hu, Y., Huang, J., L\^e, K., Nualart, D. and Tindel, S.  \, 
Stochastic heat equation with rough dependence in space. 
{\it Ann. Probab. \bf 45}  (2017),   4561-4616. 

\bibitem{HHLNT2} Hu, Y., Huang, J., L\^e, K., Nualart, D. and Tindel, S.  \, 
Parabolic Anderson model with rough dependence in space. 
{\it Computation and combinatorics in Dynamics, Stochastics and Control}, 477-498, Abel Symp., 13, Springer, Cham, 2018.

\bibitem{HHNT}Hu, Y., Huang, J., Nualart, D. and Tindel, S. \, 
Stochastic heat equations with general  multiplicative Gaussian noises: H\"older continuity and intermittency. 
{\it Electron. J. Probab.} 20 (2015), no. 55, 50 pp.

\bibitem{hule2019} 
Hu, Y. and Le, K.   \, 
Joint  H\"older continuity of parabolic Anderson model. 
{\it Acta Math Sci. \bf l 39} (2019),  764-780. 
 
\bibitem{HLT} Hu, Y.,  Liu, Y. and Tindel, S.  \, 
On the necessary and sufficient conditions to solve a heat equation 
with general additive Gaussian noise. 
 {\it Acta Math Sci.}  \textbf{39} (2019),  669-690. 

\bibitem{HLN}  Hu, Y., Lu, F., Nualart, D. \, 
 Feynman-Kac formula for the heat equation driven by fractional noise with Hurst parameter $H<1/2$. {\it Ann. Probab.}  \textbf{40}
 (2012), no. 3, 1041-1068.

 

\bibitem{HN} Hu, Y. and  Nualart, D. \, 
Stochastic heat equation driven by fractional noise and local time. {\it Probab. Theory Related Fields } \textbf{143} (2009), no. 1-2, 285-328.

\bibitem{HNS} Hu, Y.,  Nualart, D. and  Song, J.   \, 
Feynman-Kac formula for heat equation driven by fractional white noise.
{\it Ann.   Probab. \bf  39} (2011), no. 1, 291-326.

\bibitem{kilbas} Kilbas, A. A.,   Srivastava, H. M. and Trujillo, J. J.   \, 
{\it Theory and Applications of Fractional Differential Equations}.
North-Holland Mathematics Studies, {\bf 204}. Elsevier Science B.V., Amsterdam, 2006. 
   
 

\end{thebibliography}
\end{document}